\documentclass[11pt]{amsart}
\usepackage{amssymb,latexsym}
\input xypic
  \xyoption{all}

\usepackage{amsmath, amsthm, amsfonts}

\usepackage{pifont,pxfonts,txfonts}
\usepackage{yfonts}
\usepackage{bbm}
\usepackage{calrsfs}
\usepackage{empheq}
\usepackage{color}
\usepackage[colorlinks,linkcolor=red,anchorcolor=blue,citecolor=green]{hyperref}
\usepackage{amsmath, amstext, amsbsy, amscd}
\usepackage[mathscr]{eucal}

\usepackage{times}

  \usepackage{amsmath, amsthm, amsfonts,extarrows}

\newtheorem{theorem}{Theorem}[section]
\newtheorem{proposition}[theorem]{Proposition}
\newtheorem{lemma}[theorem]{Lemma}
\newtheorem{remark}[theorem]{Remark}

\newtheorem{corollary}[theorem]{Corollary}

\newcommand{\cO}{{\mathcal O }}

   \newcommand{\ba}{\begin{eqnarray}}
   \newcommand{\na}{\end{eqnarray}}
   \newcommand{\ban}{\begin{eqnarray*}}
   \newcommand{\nan}{\end{eqnarray*}}

  \newcommand{\Z}{{\mathbb Z}}

\newcommand{\bP}{{\mathbb P}}
\newcommand{\bQ}{{\mathbb Q}}
\newcommand{\bZ}{{\mathbb Z}}

  \renewcommand{\a}{\alpha}

  \newcommand{\<}{\langle}
  \renewcommand{\>}{\rangle}

\begin{document}

\title[Blow-up formulae of high genus GW invariants]{Blow-up formulae of high genus Gromov-Witten invariants in dimension six}

\author[Weiqiang He]{Weiqiang He$^1$}
  \address{Department of Mathematics\\  Sun Yat-Sen University\\
                        Guangzhou,  510275\\ China }
  \email{btjim1987@gmail.com}
\thanks{${}^1$Partially supported by China Scholarship Council}

  \author[Jianxun Hu] {Jianxun Hu$^2$}
  \address{Department of Mathematics\\  Sun Yat-Sen University\\
                        Guangzhou,  510275\\ China }
  \email{stsjxhu@mail.sysu.edu.cn}
\thanks{${}^2$Partially supported by NSFC Grant 11228101 and 11371381}

\author[Hua-Zhong Ke]{Hua-Zhong Ke}
  \address{Department of Mathematics\\  Sun Yat-Sen University\\
                        Guangzhou,  510275\\ China }
  \email{kehuazh@mail.sysu.edu.cn}

\author[Xiaoxia Qi]{Xiaoxia Qi}
  \address{ Sino-French Institute of Nuclear and Technology\\ Sun Yat-sen University\\ Tang Jia Wan, Zhuhai 519082\\ China }
\email{qxiaoxia@mail.sysu.edu.cn}

\begin{abstract}

   Using the degeneration formula and absolute/relative correspondence, one studied the change of Gromov-Witten invariants under blow-up for six dimensional symplectic manifolds and obtained closed blow-up formulae for high genus Gromov-Witten invariants. Our formulae also imply some relations among generalized BPS numbers introduced by Pandharipande.

\end{abstract}

\date{\today}
\maketitle
\tableofcontents

\section{Introduction}

Gromov-Witten invariants count stable pseudo-holomorphic curves in a symplectic manifold. The Gromov-Witten invariants for semi-positive symplectic manifolds were first defined by Ruan \cite{R1} and Ruan-Tian \cite{RT1, RT2}. Gromov-Witten invariants can be applied to define a quantum product on the cohomology groups of a symplectic manifold in \cite{RT1} and have many applications in symplectic geometry and symplectic topology, see \cite{MS} and references therein. Using the virtual moduli cycle technique, Li-Tian \cite{LT1} defined the Gromov-Witten invariants purely algebraically for smooth projective varieties. During last two decades, there were a great deal of activities to remove the semi-positivity condition, see \cite{B, FO, R2, S, LT2}. After its mathematical foundation was established, the study of Gromov-Witten theory focused on its computation and applications. We now know a lot about genus zero invariants of, say, toric manifolds, homogeneous spaces, etc. Some of the higher genus computations have also been done, but the understanding of higher genus Gromov-Witten invariants is  still far from complete.

The computation of the Gromov-Witten invariants is known to be a difficult problem in geometry and physics. There are two major techniques: the degeneration formula and localization. Li-Ruan \cite{LR} first obtained the degeneration formula, see \cite{IP} for a different version and \cite{Li} for an algebraical version. It used to be applied to the situations that a symplectic or Kahler manifold $X$ degenerates into a union of two pieces $X^\pm$ glued along a common divisor $Z$. The idea of degeneration formula is to express the Gromov-Witten invariants of $X$ in terms of relative Gromov-Witten invariants of the pairs $(X^\pm, Z)$. Localization played a very important role in the computation of Gromov-Witten invariants. Kontsevich \cite{Ko2} first introduced this technique into this field, then Givental \cite{Gi} and Lian-Liu-Yau \cite{LLY} applied this technique to prove the mirror theorem in the genus zero case. So far the computation of high genus invariants is still a difficult task. The difficulty is that the localization technique often transfers the computation of high genus invariants into that of some Hodge integrals over $\bar{\mathcal M}_{g,n}$, which so far one does not have effective methods to compute. To obtain some general structures or close formulae of Gromov-Witten theory in many applications, we degenerate a symplectic or Kahler manifold into two toric relative pairs $(X^\pm,Z)$ and then use the localization technique to compute the associated relative invariants, see \cite{HLR, MP}. The combination of the degeneration technique and localization technique has proven to be very powerful.

Ruan \cite{R3} speculated that there should be a deep relation between quantum cohomology and birational geometry. The birational symplectic geometry program requires a thorough understanding of blow-up type formula of Gromov-Witten invariants and quantum cohomology, because blow-up is the elementary birational surgery. Actually, it is rare to be able to obtain a general blow-up formula. For the last twenty years, only a few limited case were known, see \cite{H1,H2,G}. Hu-Li-Ruan \cite{HLR} studied the change of Gromov-Witten invariants under blow-up and obtained a blow-up correspondence of absolute/relative Gromov-Witten invariants. The second named author \cite{H1,H2} obtained some blow-up formulae for genus zero Gromov-Witten invariants. In this paper, we try to apply the degeneration formula to study the change of Gromov-Witten invariants under blow-ups and generalize a genus zero formula in \cite{H1} to all genera case in dimension six.

Throughout this paper, let $X$ be a connected, closed, smooth symplectic manifold of real dimension six, and $p:\tilde{X}\rightarrow X$ the natural projection of the symplectic blow-up $\tilde{X}$ of $X$ along a connected smooth symplectic submanifold of $X$. Let $E$ be the exceptional divisor of the blow-up, and $e\in H_2(\tilde X,\bZ)$ the class of a line in the fiber of $E$. Note that $p$ induces a natural injection via 'pullback' of $2$-cycles
$$
p^!=PD_{\tilde X}\circ p^*\circ PD_X:H_2(X,\bZ)\rightarrow H_2(\tilde X,\bZ),
$$
where the image of $p^!$ is the subset of $H_2(\tilde X,\bZ)$ consisting of $2$-cycles having intersection number zero with $E$.

We first consider blow-up at a point. Given a nonzero class $A\in H_2(X,\bZ)$, from the viewpoint of geometry, we could express the condition of counting curves with homology class $A$ passing through a generic point in $X$ in two ways: adding a point class, or blowing up $X$ at the point and counting curves in $\tilde X$ with homology class $p^!A-e$. One would expect that the two methods give the same Gromov-Witten invariants, which was proved by the second named author \cite{H1} in all dimensions for $g=0$, and by the fourth named author \cite{Q} in real dimension four for all genera. In this paper, we study the dimension six case for all genera:
\begin{theorem}\label{pt1}
Let $p:\tilde X\rightarrow X$ be the blow-up at a point. Suppose  that $\alpha_1,\cdots,\alpha_m\in H^{>0}(X, {\mathbb Q})$, $1\leq i \leq m$, and $d_1,\cdots,d_m\in\bZ_{\geqslant 0}$. Then for nonzero $A\in H_2(X, {\mathbb Z})$ and $g\geq 0$, we have
$$
    \langle[pt],\tau_{d_1}\alpha_1,\cdots,\tau_{d_m}\alpha_m\rangle^X_{g, A}
 = \sum_{g_1+g_2=g}\frac{(-1)^{g_1}\cdot 2}{(2g_1 + 2)!}\langle \tau_{d_1}p^*\alpha_1,\cdots, \tau_{d_m}p^*\alpha_m\rangle^{\tilde{X}}_{g_2, p^!A-e}.
$$
\end{theorem}

\begin{theorem}\label{pt2}
Under the same assumptions as in Theorem \ref{pt1}, we have
\begin{eqnarray*}
   & &\langle\tau_1[pt],\tau_{d_1}\alpha_1,\cdots,\tau_{d_m}\alpha_m\rangle^X_{g, A} \\
   &=&\sum_{g_1+g_2=g}\frac{(-1)^{g_1}}{(2g_1 + 1)!}\langle -E^2,\tau_{d_1}p^*\alpha_1,\cdots, \tau_{d_m}p^*\alpha_m\rangle^{\tilde{X}}_{g_2, p^!A-e}.
\end{eqnarray*}
\end{theorem}

Through studying the proof of Theorem \ref{pt2} carefully, we obtain the following result, which seems to be nontrivial when compared with divisor equation and dilaton equation.

\begin{theorem}\label{pt3}
Under the same assumptions as in Theorem \ref{pt1}, we have
\ban
&&\langle\tau_1E,\tau_{d_1}p^*\alpha_1,\cdots,\tau_{d_m}p^*\alpha_m\rangle^{\tilde X}_{g,p^!A-e}\\ &=&3\langle -E^2,\tau_{d_1}p^*\alpha_1,\cdots, \tau_{d_m}p^*\alpha_m\rangle^{\tilde{X}}_{g, p^!A-e}\\
&&\qquad-2\sum_{g_1+g_2=g}\frac{(-1)^{g_1}}{(2g_1 + 1)!}\langle -E^2,\tau_{d_1}p^*\alpha_1,\cdots, \tau_{d_m}p^*\alpha_m\rangle^{\tilde{X}}_{g_2, p^!A-e}.
\nan
\end{theorem}

We also consider the blow-up along a curve.

\begin{theorem}\label{curve}
Let $p:\tilde X\rightarrow X$ be the blow-up along a smooth curve $C$ with $\int_Cc_1(X)>0$. Suppose  that $\alpha_1,\cdots,\alpha_m\in H^{>2}(X, {\mathbb Q})$, $1\leq i \leq m$, support away from the curve $C$, and $d_1,\cdots,d_m\in\bZ_{\geqslant 0}$. Then for nonzero $A\in H_2(X, {\mathbb Z})$ and $g\geq 0$, we have
$$
 \langle[C],\tau_{d_1}\alpha_1,\cdots,\tau_{d_m}\alpha_m\rangle^X_{g, A} = \sum_{g_1+g_2=g}\frac{(-1)^{g_1}}{(2g_1 + 1)!\cdot 2^{2g_1}}\langle \tau_{d_1}p^*\alpha_1,\cdots, \tau_{d_m}p^*\alpha_m\rangle^{\tilde{X}}_{g_2, p^!A-e}.
$$
\end{theorem}

The above blow-up formulae relate Gromov-Witten invariants of $X$ and those of $\tilde{X}$ in a nontrivial way. Theorem \ref{pt1} and \ref{curve} imply the following simple relations among generalized BPS numbers $n^X_{g,A}(\alpha_1, \dots, \alpha_m)$ introduced by Pandharipande \cite{P1,P2}.

\begin{proposition}\label{BPS}
Suppose that $\alpha_1,\cdots,\alpha_m\in H^{>2}(X,\bQ)$, $A\in H_2(X,\bZ)$ is nonzero and $g\in\bZ_{\geqslant 0}$.
\begin{itemize}
\item[(a)]If $p:\tilde X\rightarrow X$ is the blow-up at a point, then we have
$$
n_{g,A}^X([pt],\alpha_1, \cdots, \alpha_m) = n_{g, p^!A-e}^{\tilde{X}}(p^*\alpha_1,\cdots,p^*\alpha_m).
$$
\item[(b)]If $p:\tilde X\rightarrow X$ is the blow-up along a smooth curve $C$ with $\int_Cc_1(X)>0$, then we have
$$
n_{g,A}^X([C],\alpha_1, \cdots, \alpha_m) = n_{g, p^!A-e}^{\tilde{X}}(p^*\alpha_1,\cdots,p^*\alpha_m).
$$
\end{itemize}
\end{proposition}

Our proof of the above blow-up formulae is inspired by the absolute/relative correspondence obtained by Hu-Li-Ruan \cite{HLR}, which is a generalization of the idea of Maulik-Pandharipande \cite{MP}. This correspondence partially describes the change of Gromov-Witten invarians under blow-ups. We first use degeneration formula to obtain comparison results between absolute and relative Gromov-Witten invariants,
and then use these comparison results to prove our blow-up formulae.

The rest of the paper is organized as follows. In Section 2, we briefly review basic materials of absolute/relative Gromov-Witten invariants and the degeneration formula. In Section 3, we consider the case of blow-up at a point and prove Theorem \ref{pt1}, \ref{pt2} and \ref{pt3}. In Section 4, we consider the case of blow-up along a smooth curve and prove Theorem \ref{curve}. In Section 5, we review the definition of generalized BPS numbers and prove Corollary \ref{BPS}.

\section{Preliminaries}

In this section, we briefly review absolute/relative Gromov-Witten invariants and the
degeneration formula and fix notations throughout. We use \cite{LR} as our general reference.

Recall that we always let $X$ be a connected compact smooth symplectic manifold of real dimension six. For $A\in H_2(X,\bZ)$, let $\overline{\mathcal M}_{g,m}(X,A)$ be
the moduli space of connected $m$-pointed stable maps to $X$ of arithmetic genus $g$ and degree $A$. Let $e_i: \overline{\mathcal M}_{g,m}(X,A)\longrightarrow X$ be the evaluation map at the $i^{th}$ marked point. The Gromov-Witten invariants of $X$ are defined as
$$
\langle \tau_{d_1}\alpha_1,\cdots,\tau_{d_m}\alpha_m\rangle^X_{g,A}
:=\int_{[\overline{\mathcal
M}_{g,m}(X,A)]^{vir}}\prod\limits_{i=1}^m\psi_i^{d_i}e_i^*\alpha_i,
$$
where $\alpha_1,\cdots,\alpha_m\in H^*(X,\bQ)$, $d_1,\cdots,d_m\in\bZ_{\geqslant 0}$, $\psi_i$ is the first Chern class of the cotangent line bundle, and $[\overline{\mathcal
M}_{g,m}(X,A)]^{vir}$ is the virtual fundamental cycle.

The degeneration formula \cite{LR, IP, Li} provides a rigorous
formulation about the change of Gromov-Witten invariants under
semi-stable degeneration, or symplectic cutting. The formula
relates  the absolute Gromov-Witten invariant of $X$ to the
relative Gromov-Witten invariants of two smooth pairs.

Now we recall the relative invariants of a smooth relative pair $(X,Z)$
with $Z\hookrightarrow X$ a connected smooth symplectic divisor. Let $A\in
H_2(X,{\mathbb Z})$ with $A\cdot Z\geqslant 0$, and $\mu$ a partition of
$A\cdot Z$. We
customarily use relative graphs to describe the topological
type of relative stable maps. A connected relative graph $\Gamma =
(g,m,A,\mu)$ is defined to be a connected decorated graph
consisting of the following data:
\begin{enumerate}
\item[(1)] a vertex decorated by $A$ and
genus $g$;
\item[(2)] $m$ tails with no decoration;
\item[(3)] $\ell(\mu)$ tails decorated by entries of $\mu$.
\end{enumerate}
A connected relative stable map has topological type $\Gamma$ if it has arithmetic genus $g$, degree $A$, $m$ absolute marked points and $\ell(\mu)$ relative marked points with contact order given by $\mu$. Let $\overline{\mathcal M}_\Gamma (X,Z)$ be the moduli space of connected
relative stable maps with topological type $\Gamma$. Let $e_i: \overline{\mathcal M}_{\Gamma}(X,Z)\longrightarrow X$ be the evaluation map at the $i^{th}$ absolute marked point, and $e_j^Z:\overline{\mathcal M}_{\Gamma}(X,Z)\longrightarrow Z$ the evaluation map at the $j^{th}$ relative marked point. The relative Gromov-Witten invariants of $(X,Z)$ are of the form
$$
\langle \tau_{d_1}\alpha_1,\cdots,\tau_{d_m}\alpha_m\mid \delta_1,\cdots,\delta_{\ell(\mu)}\rangle_{\Gamma}^{X,Z} :=
\int_{[\overline{\mathcal M}_\Gamma (X,Z)]^{vir}}\prod_{i=1}^m\psi_i^{d_i}e_i^*\alpha_i\cdot\prod\limits_{j=1}^{\ell(\mu)}(e^Z_j)^*\delta_i,
$$
where $\alpha_1,\cdots,\alpha_m\in H^*(X,\bQ), d_1,\cdots,d_m\in\bZ_{\geqslant 0}, \delta_1,\cdots,\delta_{\ell(\mu)}\in H^*(Z,\bQ)$, and $[\overline{\mathcal M}_\Gamma (X,Z)]^{vir}$ is the virtual fundamental cycle of dimension:
$$
\dim [\overline{\mathcal M}_\Gamma (X,Z)]^{vir}=2\int_Ac_1(X) +2m+2\ell (\mu)-2|\mu|.
$$
The relative invariants with disconnceted domains are defined by the usual product rule, and the invariants will be denoted by $\<\cdots|\cdots\>_\Gamma^{\bullet X,Z}$.

Next, we shall introduce the degeneration formula. Let $\pi : \chi \longrightarrow D$ be a connected, smooth symplectic manifold of real dimension eight over a
 disk $D$ such that $\chi_t = \pi^{-1}(t) \cong X $ for $t\not= 0$ and $\chi_0$ is a union of two connected compact smooth symplectic manifolds $X_1$
and $X_2$ intersecting transversally along a symplectic divisor $Z$. We write $\chi_0 = X_1\cup_Z X_2$.

Consider the natural inclusion maps
$$
  i_t: X=\chi_t \longrightarrow \chi,\,\,\,\,\,\,\,\,
  i_0:\chi_0\longrightarrow \chi,
$$
and the gluing map
$$
  g= (j_1,j_2) : X_1\coprod X_2\longrightarrow \chi_0.
$$
We have
$$
  H_2(X,\bZ)\stackrel{i_{t*}}{\longrightarrow}
  H_2(\chi,\bZ)\stackrel{i_{0_*}}{\longleftarrow}
  H_2(\chi_0,\bZ)\stackrel{g_*}{\longleftarrow} H_2(X_1,\bZ)\oplus
  H_2(X_2,\bZ),
$$
where $i_{0*}$ is an isomorphism since there exists a deformation retract from $\chi$ to $\chi_0$(see \cite{C}). Also, since the family $\chi\longrightarrow D$ comes from a trivial family, it follows that each $\alpha\in H^*(X,\bQ) $ has
global liftings such that the restriction $\alpha(t)$ on $\chi_t$ is defined for all $t$.

Fix a basis $\{\delta_i\}$ of $H^*(Z,\bQ)$ and denote by $\{\delta^i\}$ its dual basis. The degeneration formula expresses the absolute invariants of $X$ in terms of the relative invariants of the two smooth pairs $(X_1,Z)$ and $(X_2,Z)$:
\begin{eqnarray}\label{degeneration formula}
 & & \langle \tau_{d_1}\alpha_1,\cdots,\tau_{d_m}\alpha_m\rangle^X_{g,A}\\
& = & \sum_\mu {\mathfrak z}(\mu)\sum\limits_{i_1,\cdots,i_{\ell(\mu)}}\sum_{\eta\in\Omega_\mu}\langle
  \tau_{d_{i^-_1}}j_1^*\alpha_{i^-_1}(0),\cdots,\tau_{d_{i^-_{k_1}}}j_1^*\alpha_{i^-_{k_1}}(0)\mid \delta_{i_1},\cdots,\delta_{i_{\ell(\mu)}}\rangle^{\bullet
  X_1,Z}_{\Gamma_1}\nonumber\\
  & &\,\, \cdot\,\,\langle \tau_{d_{i^+_1}}j_2^*\alpha_{i^+_1}(0),\cdots,\tau_{d_{i^+_{k_2}}}\alpha_{i^+_{k_2}}(0)\mid{\delta}^{i_1},\cdots,\delta^{i_{\ell(\mu)}}\rangle^{\bullet X_2,Z}_{\Gamma_2}, \nonumber
\end{eqnarray}
where ${\mathfrak z}(\mu)=|\mbox{Aut} \mu|\prod\limits_{i=1}^{\ell(\mu)}\mu_i$, and $\eta = (\Gamma_1,\Gamma_2, I)$ is an admissible triple, which consists of (possibly disconnected) topological types $\Gamma_1,\Gamma_2$ with the same partition $\mu$ under the identification $I$ of relative marked points, satisfying the following requirements:
\begin{itemize}
\item[(1)] the gluing of $\Gamma_1$ and $\Gamma_2$ under $I$ is connected;
\item[(2)] let $g_i$ be the total genus of $\Gamma_i$, and we have
$g=g_1+g_2+\ell(\mu)+1-|\Gamma_1| -|\Gamma_2|$, where $|\Gamma_i|$ is the number of connected components of $\Gamma_i$;
\item[(3)] let $A_i\in H_2(X_i,\bZ)$ be the total degree of $\Gamma_i$, and we have $i_{t*}A=i_{0*}(j_{1*}A_1+j_{2*}A_2)$ and $|\mu|=A_1\cdot Z=A_2\cdot Z$;
\item[(4)] the absolute marked points of $\Gamma_1,\Gamma_2$ are indexed by $\{i_1^-,\cdots,i_{k_1}^-\}$ and $\{i_1^+,\cdots$, $i_{k_2}^+\}$ respectively, the disjoint union of which is exactly $\{1,2,\cdots, m\}$.
\end{itemize}
We denote by $\Omega_\mu$ the equivalence class
of all admissible triples with fixed partition $\mu$. For $\eta\in\Omega_\mu$ having nonzero contribution in the degeneration formula, we have the following important dimension constraint ({\bf Theorem 5.1 in \cite{LR}}):
\begin{equation}\label{dimension}
       \dim \overline{\mathcal M}_{\Gamma_1}(X_1,Z) + \dim \overline{\mathcal M}_{\Gamma_2}(X_2,Z) = \dim \overline{\mathcal M}_{g,m}(X,A) + 4\ell(\mu).
\end{equation}

\begin{remark}\label{symplectic cut}
Symplectic cutting is a kind of surgery in symplectic geometry which is suitable for the above degeneration formula (see \cite{LR}). Suppose that $X_0\subset X$ is an open codimension zero submanifold with Hamiltonian $S^1$-action. Let $H: X_0\longrightarrow {\mathbb R}$ be a Hamiltonian function with $0$ as a regular value. If $H^{-1}(0)$ is a separating hypersurface of $X_0$, then we obtain two connected manifolds $X^\pm_0$ with boundary
$\partial X^\pm_0= H^{-1}(0)$, where the $+$ side corresponds to $H<0$. Suppose further that $S^1$ acts freely on $H^{-1}(0)$. Then the symplectic reduction $Z=H^{-1}(0)/S^1$ is canonically a symplectic manifold. Collapsing the $S^1$-action on $\partial X^\pm = H^{-1}(0)$, we obtain two closed smooth manifolds $\bar{X}^\pm$ containing respectively real codimension 2 submanifolds $Z^\pm =Z$ with opposite normal bundles. Furthermore
$\bar{X}^\pm$ admits a symplectic structure $\bar{\omega}^\pm$ which agrees with the restriction of $\omega$ away from $Z$, and whose restriction to $Z^\pm$ agrees with the canonical symplectic structure $\omega_Z$ on $Z$ from symplectic reduction. The pair of symplectic manifolds $(\bar{X}^\pm,\bar{\omega}^\pm)$ is called the symplectic cut of $X$ along $H^{-1}(0)$.

Suppose that $Y\subset X$ is a submanifold of $X$ of codimension $2k$. Denote by $N_Y$ the normal bundle. By the symplectic neighborhood theorem ,and by possibly taking a smaller $\epsilon_0$, a tubular neighborhood ${\mathcal N}_{\epsilon_0}(Y)$ of $Y$ in $X$ is symplectomorphic to the disc bundle $N_Y(\epsilon_0)$ of $N_Y$. Denote by $\phi: {\mathcal N}_{\epsilon_0}(Y) \longrightarrow N_Y(\epsilon_0)$ be such a symplectomorphism. Consider the Hamiltonian $S^1$-action on $X_0={\mathcal N}_{\epsilon_0}(Y)$ by complex multiplication. Fix $\epsilon$ with $0<\epsilon <\epsilon_0$ and consider the moment map
$$
                 H(u) = |\phi(u)|^2 - \epsilon, \,\,\,\,\, u\in {\mathcal N}_Y(\epsilon_0),
$$
where $|\phi (u) |$ is the norm of $\phi(u)$ considered as a vector in a fiber of the Hermitian bundle $N_Y$. We cut $X$ along $H^{-1}(0)$ to obtain two closed symplectic manifolds $\bar{X}^\pm$. Notice that $\bar{X}^+\cong
{\mathbb P}_Y(N_Y\oplus\cO_Y)$. $\bar{X}^-$ is called the blow-up of $X$ along $Y$, denoted by $\tilde{X}$.
\end{remark}

\section{Formulae for Blow-up at a point}

In this section, we prove Theorem \ref{pt1}, \ref{pt2} and \ref{pt3}. We always assume that total degrees of insertions match the virtual dimension of the moduli spaces, since otherwise the required equalities are trivial.

First of all, we will divide the proof of Theorem \ref{pt1} into some comparison theorems of Gromov-Witten invariants as follows.

\begin{lemma}\label{ptlemma1}
Under the same assumptions as in Theorem \ref{pt1}, we have
\begin{eqnarray}\label{3-1}
     & & {\langle[pt],\tau_{d_1}\alpha_1,\cdots,\tau_{d_m}\alpha_m\rangle}^X_{g,A}\nonumber\\
     & & \nonumber\\
     & = &  \sum_{g^++g^-=g} {\langle [pt]|[pt]\rangle}^{\mathbb P^3,H}_{g^+,L,(1)}
     {\langle\tau_{d_1}p^*\alpha_1,\cdots,\tau_{d_m}p^*\alpha_m|\mathbbm 1\rangle}^{\tilde X,E}_{g^-,p^!A-e,(1)},
\end{eqnarray}
where $H$ is the hyperplane at infinity, and $L\in H_2(\mathbb P^3;\Z)$ is the class of a line.
\end{lemma}

\begin{proof}
We first perform the symplectic cutting along a point as in Remark \ref{symplectic cut}. Here we have assumed that the class $[pt]$ has support in $X^+$ and $\alpha_i$ has support in $X^-$. By the degeneration formula \eqref{degeneration formula},we have
\begin{eqnarray*}
    &   &\langle[pt],\tau_{d_1}\alpha_1,\cdots,\tau_{d_m}\alpha_m\rangle^X_{g,A}\\
    & = &\sum {\mathfrak z}(\mu)\langle [pt]|\delta_{j_1},\cdots,\delta_{j_{\ell(\mu)}}\rangle^{\mathbb P^3,H}_{g^+,A^+,\mu}\\
    &  & \cdot\langle\tau_{d_1}p^*\alpha_1,\cdots,\tau_{d_m}p^*\alpha_m|\delta^{j_1},\cdots,\delta^{j_{\ell(\mu)}}\rangle^{\tilde X,E}_{g^-,A^-,\mu}.
\end{eqnarray*}
 By our assumption that total degrees of insertions match the virtual dimension of the moduli space, we have
$$
     \dim \bar{\mathcal M}_{g,m+1}(X,A) = \sum_{i=1}^m \deg\alpha_i+2\sum_{i=1}^md_i+6.
$$

Suppose that $(\Gamma^+,\Gamma^-)$ has nonzero contribution in the degeneration formula. Then
\begin{eqnarray*}
   \dim \bar{\mathcal M}_{\Gamma^+}(\mathbb P^3,H) & = &  2\int_{A^+}c_1(\bP^3)+2+2\ell(\mu)-2|\mu|,\\
   \dim \bar{\mathcal M}_{\Gamma^-}(\tilde X,E) & = &\sum_{i=1}^m \deg\alpha_i+2\sum_{i=1}^md_i+\sum\limits_{i=1}^{\ell(\mu)}\deg\delta^{j_i}.
\end{eqnarray*}
So by the dimension constraint \eqref{dimension},
$$
     \frac{1}{2}\sum\limits_{i=1}^{\ell(\mu)}deg\delta^{j_i}+\int_{A^+}c_1(\bP^3)-|\mu|=2+\ell(\mu).
$$
Note that $A^+\cdot H=|\mu|$, and hence $A^+=|\mu|L$, which implies that
\ban
\int_{A^+}c_1(\bP^3)=4|\mu|.
\nan
Now the dimension constraint becomes
\ban
\frac 12\sum\limits_{i=1}^{\ell(\mu)}deg\delta^{j_i}+3|\mu|=2+\ell(\mu).
\nan
So the dimension constraint holds only if
\ban
\mu=(1),\quad deg\delta^{j_1}=0,
\nan
which implies the required equality.
\end{proof}

\begin{lemma}\label{ptlemma2}
Under the same assumptions as in Theorem \ref{pt1}, we have
\begin{eqnarray}\label{3-2}
& & {\langle \tau_{d_1}p^*\alpha_1,\cdots ,\tau_{d_m}p^*\alpha_m\rangle}^{\tilde X}_{g,p^!A-e}\nonumber \\
&=& \sum_{g^++g^-=g}{\langle \,\,|[pt]\rangle}_{g^+,F,(1)}^{\tilde{\mathbb P}^3,H} \\
& & \cdot{\langle \tau_{d_1}p^*\alpha_1,\cdots, \tau_{d_m}p^*\alpha_m|\mathbbm 1\rangle}^{\tilde X,E}_{g^-,p^!A-e,(1)},\nonumber
\end{eqnarray}
where $F\in H_2(\tilde\bP^3,\bZ)$ is the class of a fiber in $\tilde{\mathbb P}^3\cong{\mathbb P}_{\bP^2}(\mathcal O\oplus\mathcal O(-1))$.
\end{lemma}

\begin{proof}
We perform symplectic cutting along $E$ in $\tilde X$  as in Remark \ref{symplectic cut}. Here we also assumed that the class $p^*\alpha_i$ has support away from $E$.
By the degeneration formula (\ref{degeneration formula}), we have
\begin{eqnarray}\label{3-3}
& & {\langle \tau_{d_1}p^*\alpha_1, \cdots, \tau_{d_m}p^*\alpha_m\rangle}^{\tilde X}_{g,p^!A-e}\nonumber\\
&=&\sum {\mathfrak z}(\mu){\langle \,\,|\delta_{j_1},\cdots,\delta_{j_{\ell(\mu)}}\rangle}
         _{g^+,(p!(A)-e)^+,\nu}^{\tilde{\mathbb P}^3,H} \\
& & \cdot{\langle \tau_{d_1}p^*\alpha_1,\cdots, \tau_{d_m}p^*\alpha_m|\delta^{j_1},\cdots,\delta^{j_{\ell(\mu)}}\rangle}
         ^{\tilde X,E}_{g^-,(p^!A-e)^-,\mu}.\nonumber
\end{eqnarray}
 By our assumption that degrees match the virtual dimension, we have
\ban
\dim_{\mathbb C}\bar{\mathcal M}_{g,m}(\tilde X,p^!A-e)&=&\frac 12\sum_{i=1}^m deg\alpha_i+\sum_{i=1}^md_i.
\nan

Suppose that a term with $(\Gamma^+,\Gamma^-)$ has nonzero contribution in RHS of the degeneration formula (\ref{3-3}). Then
$$
 \dim_{\mathbb C}\bar{\mathcal M}_{\Gamma^+}(\tilde{\mathbb P}^3,H) = \int_{(p^!A-e)^+}c_1(\tilde{\bP}^3)+\ell(\mu)-|\mu|,
$$
$$
 \dim_{\mathbb C}\bar{\mathcal M}_{\Gamma^-}(\tilde X,E) = \frac 12\sum_{i=1}^mdeg\alpha_i+\sum_{i=1}^md_i+\frac 12\sum_{i=1}^{\ell(\mu)}deg\delta^{j_i}.
$$
So by the dimension constraint \eqref{dimension},
\ban
\frac 12\sum\limits_{i=1}^{\ell(\mu)}deg\delta^{j_i}+\int_{(p^!A-e)^+}c_1(\tilde{\bP}^3)-|\mu|=\ell(\mu).
\nan

Let $L\in H_2(\tilde\bP^3,\bZ)$ be the class of the total transform of a line in $\bP^3$. Then we have the following natural decomposition
\ban
H_2(\tilde\bP^3,\bZ)=\bZ F\oplus\bZ L.
\nan
We have the following constraints for $(p^!A-e)^+$:
$$
    {(p!(A)-e)}^+\cdot H = |\mu|,   \,\,\,\,\,    (p^!(A)-e)^+\cdot E = 1.
$$
So we have $(p^!A-e)^+=F+(|\mu|-1)L$, and hence $\int_{(p^!A-e)^+}c_1(\tilde\bP^3)=4|\mu|-2$. Now the dimension constraint becomes
\ban
\frac 12\sum\limits_{i=1}^{\ell(\mu)}deg\delta^{j_i}+3|\mu|=2+\ell(\mu).
\nan
So the dimension constraint holds only if
\ban
\mu=(1),\quad deg\delta^{j_1}=0,
\nan
which implies the required equality.
\end{proof}

Using the above comparison results we may obtain the following absolute/relative correspondence for Gromov-Witten invariants under blow-up.
\begin{lemma}\label{ptlemma3}
Under the same assumptions as in Theorem \ref{pt1}, denote $\langle[pt],\tau_{d_1}\a_1,$ $\cdots,\tau_{d_m}\a_m\rangle^X_{g,A}$ and
$\langle \tau_{d_1}p^*\a_1,\cdots,\tau_{d_m}p^*\a_m\rangle^{\tilde X}_{g,p^!A-e}$ by
$H_g$ and $P_g$ respectively. Then
\begin{equation}\label{3-4}
      H_g=\sum_{g_1+g_2=g}C_{g_1}P_{g_2},
\end{equation}
where $C_g$'s can be determined by relative invariants ${\langle [pt]|[pt]\rangle}^{\mathbb P^3,H}_{g,L,(1)}$
 and ${\langle \,\,|[pt]\rangle}_{g,F,(1)}^{\tilde{\mathbb P}^3,H}$.
\end{lemma}

\begin{proof}
Denote $\langle \tau_{d_1}p^*\alpha_1,\cdots,\tau_{d_m}p^*\alpha_m|\mathbbm 1\rangle
^{\tilde X,E}_{g,p!(A)-e,(1)}$, ${\langle [pt]|[pt]\rangle}^{\mathbb P^3,H}_{g,L,(1)}$ and ${\langle \,\,|[pt]\rangle}_{g,F,(1)}^{\tilde{\mathbb P}^3,H}$ by $K_g$, $I_g$ and $J_g$ respectively.
Then for $g\geq 0$, we may rewrite our comparison results (\ref{3-1}) and (\ref{3-2}) as
\begin{eqnarray*}
H_g&=&I_gK_0+I_{g-1}K_1+\cdots+I_0K_g\\
\\
P_g&=&J_gK_0+J_{g-1}K_1+\cdots+J_0K_g ,
\end{eqnarray*}
or in matrix form
$$
      \left(\begin{array}{l}
      H_0\\H_1\\\vdots\\H_g \end{array}\right)=
      \left(\begin{array}{llll}
      I_0&&\multicolumn{2}{c}{0}\\
      I_1&I_0&&\\
      \vdots&&\ddots&\\
      I_g&I_{g-1}&\cdots&I_0
      \end{array}\right)
      \left(\begin{array}{l}
      K_0\\K_1\\\vdots\\K_g \end{array}\right),
$$

$$
       \left(\begin{array}{l}
       P_0\\P_1\\\vdots\\P_g \end{array}\right)=
       \left(\begin{array}{llll}
       J_0&&\multicolumn{2}{c}{0}\\
      J_1&J_0&&\\
      \vdots&&\ddots&\\
      J_g&J_{g-1}&\cdots&J_0
      \end{array}\right)
      \left(\begin{array}{l}
     K_0\\K_1\\\vdots\\K_g \end{array}\right).
$$
This is a special form of absolute/relative correspondence for Gromov-Witten invariants ({\bf Theorem 5.15 in \cite{HLR}}). In particular, $I_0\neq 0$ and $J_0\neq 0$, which implies that both matrices with entries $I_g$ and $J_g$ are invertible (one can also use virtual localization \cite{GP} to check that $I_0=J_0=1$).
Write
$$\left(\begin{array}{llll}
       C_0&&\multicolumn{2}{c}{0}\\
       C_1&C_0&&\\
       \vdots&&\ddots&\\
       C_g&C_{g-1}&\cdots&C_0
        \end{array}\right)=\left(\begin{array}{llll}
      I_0&&\multicolumn{2}{c}{0}\\
      I_1&I_0&&\\
      \vdots&&\ddots&\\
      I_g&I_{g-1}&\cdots&I_0
      \end{array}\right)  \left(\begin{array}{llll}
       J_0&&\multicolumn{2}{c}{0}\\
      J_1&J_0&&\\
      \vdots&&\ddots&\\
      J_g&J_{g-1}&\cdots&J_0
      \end{array}\right)^{-1},
$$and we obtain the required equality.
\end{proof}

To get Theorem \ref{pt1}, we need to compute $C_g$'s in (\ref{3-4}). A crucial observation from the proof of Lemma \ref{ptlemma3} is that $C_g$\rq{}s are determined by relative invariants ${\langle [pt]|[pt]\rangle}^{\mathbb P^3,H}_{g,L,(1)}$ and ${\langle \,\,|[pt]\rangle}_{g,F,(1)}^{\tilde{\mathbb P}^3,H}$, and are independent of the choice of $X, m, \alpha_i, A$. Therefore, to compute these universal coefficients, we may choose $X=\mathbb{P}^3, m=1, \alpha_1=[pt], A=L$. Then (\ref{3-4}) becomes
\begin{equation}\label{equ}
   \langle[pt],[pt]\rangle_{g,L}^{\mathbb P^3}=\sum_{g_1+g_2=g}C_{g_1}\cdot \langle[pt]\rangle_{g_2,F}^{\tilde{\mathbb P}^3},
\end{equation}
where $F$ is the class of a fiber in $\tilde{\mathbb P}^3 \cong {\mathbb P}_{{\mathbb P}^2}({\mathcal O}\oplus {\mathcal O}(-1))$.

To get $C_g$'s by solving the equation (\ref{equ}), we need to compute the absolute Gromov-Witten invariants $ \langle[pt],[pt]\rangle_{g,L}^{\mathbb P^3}$ and $\langle[pt]\rangle_{g_2,F}^{\tilde{\mathbb P}^3}$.
From this, we have

\begin{lemma}
\ban
\langle[pt],[pt]\rangle_{g,L}^{\mathbb P^3}&=&\frac{(-1)^g\cdot 2}{(2g+2)!},\\
\langle[pt]\rangle_{g,F}^{\tilde{\mathbb P}^3}&=&\delta_{g,0}.
\nan
\end{lemma}

These equalities can be proved either directly by virtual localization \cite{GP} or by degenerate contribution computation \cite{P2}. In fact, Theorem 3 in \cite{P2} may specialize to the case of $\bP^3$ and obtain
these invariants. Here we omit the proof.

{\bf Proof of Theorem  \ref{pt1}:} We first perform symplectic cutting at a point in $X$ and get equation (\ref{3-1}). Then we perform symplectic cutting along the exceptional divisor $E$ in $\tilde{X}$ and get (\ref{3-2}).
Finally,  we can solve the equation \eqref{equ} to get the universal coefficients
$$
                C_g=\frac{(-1)^g\cdot 2}{(2g+2)!}.
$$
This proves Theorem \ref{pt1}.

\begin{remark}
One can relax the requirement in Theorem \ref{pt1} to $m\geqslant 0$, which can be checked by going through the proof of Lemma \ref{ptlemma1}, \ref{ptlemma2} and \ref{ptlemma3}.  This also holds for Theorem \ref{pt2} and \ref{pt3}.
\end{remark}

It is illuminating to rephrase this using a genus $g$ gravitational Gromov-Witten generating function. Suppose that $T_0=1, T_1,\cdots, T_m $ is a basis for $H^*(X, {\mathbb Q})$.
We introduce supercommuting variables $t_d^j$ for $d\geq 0$ and $0\leq j\leq m$ with $\deg t^j_d = \deg T_j$. Set
$$
      \gamma = \sum_{d=0}^\infty\sum_{j=1}^mt^j_d\tau_dT_j.
$$
Define the genus $g$ gravitational Gromov-Witten generating function as
$$
               F_g^X(t_d^j) = \sum_{n=0}^\infty\sum_{A\in H_2(X, {\mathbb Z})}\frac{1}{n!}\langle \gamma^n, [pt]\rangle_{g, A}^Xq^A,
$$
$$
              F_g^{\tilde{X}}(t_d^j) = \sum_{n=0}^\infty\sum_{A\in H_2(X,{\mathbb Z})}\frac{1}{n!}\langle (p^*\gamma^n\rangle_{g,p^!(A)- e}^{\tilde{X}}q^{p^!(A)-e}.
$$
Set
$$
                  F^X(u,t_d^j) = \sum_{g\geq 0}u^{2g-2}F_g^X(t_d^j)
$$
and
$$
                 F^{\tilde{X}}(u,t_d^j) = \sum_{g\geq 0} u^{2g-2}F_g^{\tilde{X}}(t_d^j).
$$

Then from Theorem \ref{pt1}, we have

\begin{corollary}\label{generating function}
$$
          F^X(u,\gamma)=\bigg(\frac{\sin\frac{u}{2}}{\frac{u}{2}}\bigg)^2\cdot F^{\tilde{X}}(u,p^*\gamma),
$$
where we need to change the variable $q^A$ to $q^{p^!(A)-e}$.
\end{corollary}

Similar to the proof of Theorem \ref{pt1} above, we may divide the proof of Theorem \ref{pt2} into the following Lemma \ref{ptlemma4}, \ref{ptlemma5} and \ref{ptlemma6}, the proof of which is analogous to that of Lemma \ref{ptlemma1}, \ref{ptlemma2} and \ref{ptlemma3} respectively.

\begin{lemma}\label{ptlemma4}
Under the same assumptions as in Theorem \ref{pt1}, we have
\begin{eqnarray*}
     & & {\langle\tau_1[pt],\tau_{d_1}\alpha_1,\cdots,\tau_{d_m}\alpha_m\rangle}^X_{g,A}\nonumber\\
     & & \nonumber\\
     & = &  \sum_{g^++g^-=g} {\langle\tau_1[pt]|\xi\rangle}^{\mathbb P^3,H}_{g^+,L,(1)}
     {\langle\tau_{d_1}p^*\alpha_1,\cdots,\tau_{d_m}p^*\alpha_m|\xi\rangle}^{\tilde X,E}_{g^-,p^!A-e,(1)},
\end{eqnarray*}
where $H$ is the hyperplane at infinity, $L$ is the class of a line in $\bP^3$, and $\xi$ is the cohomology class of a line in $H\cong E\cong \bP^2$.
\end{lemma}

\begin{proof}
The same argument as in the proof of Lemma \ref{ptlemma1} leads to the dimension constraint
$$
     \frac{1}{2}\sum\limits_{i=1}^{\ell(\mu)}\deg \delta^{j_i} + 3|\mu| = 3 + \ell(\mu).
$$
This constraint holds only if
$$
    \mu = (1), \,\,\,\, \deg\delta_{j_1}=2.
$$
This implies Lemma \ref{ptlemma4}.
\end{proof}

\begin{lemma}\label{ptlemma5}
Under the same assumptions as in Theorem \ref{pt1}, we have
\begin{eqnarray*}
&&{\langle-E^2,\tau_{d_1}p^*\alpha_1,\cdots ,\tau_{d_m}p^*\alpha_m\rangle}^{\tilde X}_{g,p^!A-e}\\
&=&\sum_{g^++g^-=g}{\langle-E^2|\xi\rangle}_{g^+,F,(1)}^{\tilde{\mathbb P}^3,H}\cdot{\langle \tau_{d_1}p^*\alpha_1,\cdots, \tau_{d_m}p^*\alpha_m|\xi\rangle}^{\tilde X,E}_{g^-,p^!A-e,(1)},
\end{eqnarray*}
where $F\in H_2(\tilde\bP^3,\bZ)$ is the class of a fiber in $\tilde{\mathbb P}^3\cong{\mathbb P}_{\bP^2}(\mathcal O\oplus\mathcal O(-1))$, and $\xi$ is the cohomology class of a line in $H\cong E\cong\bP^2$.
\end{lemma}

\begin{proof}
    The same dimension calculation as in the proof of Lemma \ref{ptlemma1} gives rise to the dimension constraint
$$
    \frac{1}{2}\sum\limits_{i=1}^{\ell(\mu)}\deg \delta^{j_i} + 3|\mu| = 3 + \ell(\mu).
$$
This also implies
$$
    \mu = (1), \,\,\,\, \deg\delta_{j_1}=2,
$$
which proves Lemma \ref{ptlemma5}.
\end{proof}

\begin{lemma}\label{ptlemma6}
Under the same assumptions as in Theorem \ref{pt1}, denote $\langle\tau_1[pt],\tau_{d_1}\a_1$, $\cdots,\tau_{d_m}\a_m\rangle^X_{g,A}$ and
$\langle-E^2,\tau_{d_1}p^*\a_1,\cdots,\tau_{d_m}p^*\a_m\rangle^{\tilde X}_{g,p^!A-e}$ by
$H_g$ and $P_g$ respectively. Then
\begin{equation}\label{3-5}
      H_g=\sum_{g_1+g_2=g}C_{g_1}P_{g_2},
\end{equation}
where $C_g$'s can be determined by relative invariants ${\langle\tau_1[pt]|\xi\rangle}^{\mathbb P^3,H}_{g,L,(1)}$
 and ${\langle-E^2|\xi\rangle}_{g,F,(1)}^{\tilde{\mathbb P}^3,H}$. Here $\xi$ is the cohomology class of a line in $H\cong E\cong \bP^2$.
\end{lemma}

The proof of Lemma \ref{ptlemma6} is identical to that of Lemma \ref{ptlemma3} with Lemma \ref{ptlemma1} and \ref{ptlemma2} replaced by Lemma \ref{ptlemma4} and \ref{ptlemma5} respectively.

{\bf Proof of Theorem \ref{pt2}:} Similar to the proof of Theroem \ref{pt1}, we only need to compute the universal coefficients $C_g$'s in (\ref{3-5}). Similarly, we choose $X=\mathbb{P}^3, m=1, \alpha_1=[L], A=L$. Then (\ref{3-5}) becomes
\begin{equation}\label{3-6}
   \langle\tau_1[pt],[L]\rangle_{g,L}^{\mathbb P^3}=\sum_{g_1+g_2=g}C_{g_1}\cdot \langle-E^2,[L]\rangle_{g_2,F}^{\tilde{\mathbb P}^3}.
\end{equation}
By virtual localization \cite{GP}, we have
$$
\langle\tau_1[pt],[L]\rangle_{g,L}^{\mathbb P^3} = \frac{(-1)^g}{(2g+1)!},
$$
$$
\langle-E^2,[L]\rangle_{g,F}^{\tilde{\mathbb P}^3} = \delta_{g,0}.
$$
We solve (\ref{3-6}) to obtain the universal coefficients:
$$
           C_g=\frac{(-1)^g}{(2g+1)!},
$$
which gives Theorem \ref{pt2}.

In the rest of this section, we will prove Theorem \ref{pt3}. Similar argument to in the proof of Lemma \ref{3-2} and \ref{3-3}, we may prove the following Lemmas.

\begin{lemma}\label{ptlemma7}
Under the same assumptions as in Theorem \ref{pt1}, we have
\begin{eqnarray*}
&&{\langle\tau_1E,\tau_{d_1}p^*\alpha_1,\cdots ,\tau_{d_m}p^*\alpha_m\rangle}^{\tilde X}_{g,p^!A-e}\\
&=&\sum_{g^++g^-=g}{\langle\tau_1E|\xi\rangle}_{g^+,F,(1)}^{\tilde{\mathbb P}^3,H}\cdot{\langle \tau_{d_1}p^*\alpha_1,\cdots, \tau_{d_m}p^*\alpha_m|\xi\rangle}^{\tilde X,E}_{g^-,p^!A-e,(1)},
\end{eqnarray*}
where $F\in H_2(\tilde\bP^3,\bZ)$ is the class of a fiber in $\tilde{\mathbb P}^3\cong{\mathbb P}_{\bP^2}(\mathcal O\oplus\mathcal O(-1))$, and $\xi$ is the cohomology class of a line in $H\cong E\cong\bP^2$.
\end{lemma}

\begin{lemma}\label{ptlemma8}
Under the same assumptions as in Theorem \ref{pt1}, denote $\langle\tau_1E,\tau_{d_1}p^*\a_1$, $\cdots,\tau_{d_m}p^*\a_m\rangle^{\tilde X}_{g,p^!A-e}$ and
$\langle-E^2,\tau_{d_1}p^*\a_1,\cdots,\tau_{d_m}p^*\a_m\rangle^{\tilde X}_{g,p^!A-e}$ by
$H_g$ and $P_g$ respectively. Then
$$
      H_g=\sum_{g_1+g_2=g}C_{g_1}P_{g_2},
$$
where $C_g$'s can be determined by relative invariants ${\langle\tau_1E|\xi\rangle}_{g,F,(1)}^{\tilde{\mathbb P}^3,H}$
 and ${\langle-E^2|\xi\rangle}_{g,F,(1)}^{\tilde{\mathbb P}^3,H}$. Here $\xi$ is the cohomology class of a line in $H\cong E\cong \bP^2$.
\end{lemma}

{\bf Proof of Theorem \ref{pt3}:} Similar to the proof of Theorem \ref{pt2}, we choose $X=\tilde{\mathbb{P}}^3, m=1, \alpha_1=[L], A=F$ and obtain:
$$
   \langle\tau_1E,L\rangle_{g,F}^{\tilde{\mathbb P}^3}=\sum_{g_1+g_2=g}C_{g_1}\cdot \langle-E^2,L\rangle_{g_2,F}^{\tilde{\mathbb P}^3}.
$$
By virtual localization \cite{GP}, we have
$$
\langle\tau_1E,L\rangle_{g,F}^{\tilde{\mathbb P}^3} = \delta_{g,0}\cdot 3-\frac{(-1)^g\cdot 2}{(2g+1)!}.
$$
So
$$
    C_g=\delta_{g,0}\cdot 3-\frac{(-1)^g\cdot 2}{(2g+1)!},
$$
which gives Theorem \ref{pt3}.

\section{Formulae for Blow-up along a smooth curve}

In this section, we give a detailed proof of Theorem \ref{curve}. We always assume that total degrees of insertions match the virtual dimension of the moduli spaces, since otherwise the required equalities are trivial.

\begin{lemma}\label{curvelemma1}
Under the same assumptions as in Theorem \ref{curve}, we have
\begin{eqnarray*}
     &  &\langle[C],\tau_{d_1}\alpha_1,\cdots,\tau_{d_m}\alpha_m\rangle_{g,A}^X\\
     &=&\sum\limits_{g_1+g_2=g}\langle[C]|[pt]\rangle^{\bar X^+,Z}_{g_1,F,(1)}\cdot{\langle\tau_{d_1}p^*\alpha_1,\cdots,\tau_{d_m}p^*\alpha_m|\mathbbm 1\rangle}^{\tilde X,E}_{g_2,p^!A-e,(1)},
\end{eqnarray*}
where $F\in H_2(\bar X^+,\bZ)$ is the class of a line in the fiber of $\bar X^+=\bP_C(N_C\oplus\cO_C)$.
\end{lemma}

\begin{proof}
We first perform symplectic cutting along $C$ and assume that the support of $[C]$ is in $X^+$ and the support of $\alpha_i$ is away from $C$.
By the degeneration formula \eqref{degeneration formula}, we have:
\begin{eqnarray}\label{4-1}
    & &\<[C],\tau_{d_1}\alpha_1,\cdots,\tau_{d_m}\alpha_m\>_{g,A}^X\nonumber\\
    &=&\sum\limits_{}\mathfrak{z}(\mu)\<[C]|\delta_{j_1},\cdots,\delta_{j_{\ell(\mu)}}\>_{\Gamma_+}^{\bullet,\bar X^+,Z}\\
    & &\qquad\cdot\<\tau_{d_1}p^*\alpha_1,\cdots,\tau_{d_m}p^*\alpha_m|\delta^{j_1},\cdots,\delta^{j_{\ell(\mu)}}\>_{\Gamma_-}^{\bullet,\tilde X,E}.\nonumber
\end{eqnarray}
 Recall that we have assumed that
$$
      \dim \overline M_{g,m+1}(X,A) =\sum\limits_{i=1}^m deg\alpha_i+2\sum\limits_{i=1}^{m}d_i+4.
$$

Assume that A term with $(\Gamma^+,\Gamma^-)$ in RHS of (\ref{4-1}) has nonzero contribution, and then
\begin{eqnarray*}
      \dim \overline M_{\Gamma_+}(\bar X^+,Z) & = &  2\int_{A^+}c_1(X^+)+2+2\ell(\mu)-2|\mu|, \\
      \dim \overline M_{\Gamma_-}(\tilde X,E) & = &  \sum\limits_{i=1}^mdeg\alpha_i+2\sum\limits_{i=1}^{m}d_i+\sum\limits_{i=1}^{\ell(\mu)}deg\delta^{j_i}.
\end{eqnarray*}
So by the dimension constraint \eqref{dimension} for the degeneration formula, we have
\begin{eqnarray*}
\frac{1}{2}\sum\limits_{i=1}^{\ell(\mu)}deg\delta^{j_i}+\int_{A^+}c_1(X^+)-|\mu|=1+\ell(\mu).\nonumber
\end{eqnarray*}
Let $\xi^+$ be the tautological line bundle of $\bar X^+=\bP_C(N_C\oplus\cO_C)$, and we have
\begin{eqnarray*}
c_1(\bar X^+)=\pi^*c_1(X)|_C-3c_1(\xi^+),
\end{eqnarray*}
where $\pi:\bar X^+\rightarrow C$ is the natural projection. Note that $-c_1(\xi^+)$ is the Poincar\'e dual of the divisor $Z$ in $\bar X^+$. Since $|\mu|=A^+\cdot Z$, it follows that
\begin{eqnarray*}
\int_{A^+}c_1(\bar X^+)=\int_{\pi_*A^+}c_1(X)|_C+3|\mu|.
\end{eqnarray*}
Therefore, dimension constraint becomes
\begin{eqnarray*}
&&\frac{1}{2}\sum\limits_{i=1}^{\ell(\mu)}deg\delta^{j_i}+\int_{\pi_*A^+}c_1(X)|_C+2|\mu|=1+\ell(\mu).
\end{eqnarray*}
Since $m\geqslant 1$, it follows that $\mu\ne \emptyset$ by the connectedness of the stable maps to $X$, and the dimension constraint holds only if
\begin{eqnarray*}
\mu=(1),\quad\textrm{deg}\delta^{j_1}=0,\quad\int_{\pi_*A^+}c_1(X)|_C=0,
\end{eqnarray*}
which implies Lemma \ref{curvelemma1}.
\end{proof}

\begin{lemma}\label{curvelemma2}
Under the same assumptions as in Theorem \ref{curve}, we have
\begin{eqnarray*}
    &  &\langle\tau_{d_1}p^*\alpha_1,\cdots,\tau_{d_m}p^*\alpha_m\rangle_{g,p^!A-e}^{\tilde X}\\
    & = &\sum\limits_{g_1+g_2=g}\langle|[pt]\rangle^{\bar{\tilde X}^+,Z}_{g_1,F,(1)}{\langle\tau_{d_1}p^*\alpha_1,\cdots,\tau_{d_m}p^*\alpha_m|\mathbbm 1\rangle}^{\tilde X,E}_{g_2,p^!A-e,(1)},
\end{eqnarray*}
where $F\in H_2(\tilde X^+,\bZ)$ is the class of a line in the fiber of $\bar{\tilde X}^+=\bP_E(N_E\oplus\cO_E)$.
\end{lemma}

\begin{proof}
We first degenerate $\tilde X$ along $E$, and assume that the support of $p^*\alpha_i$ is away from $E$. By the degeneration formula \eqref{degeneration formula}, we have
\begin{eqnarray}\label{4-2}
   &   &{\langle E,\tau_{d_1}p^*\alpha_1,\cdots,\tau_{d_m}p^*\alpha_m\rangle}^{\tilde X}_{g,p^!A-e}\nonumber\\
   & = &\sum\limits_{}\mathfrak{z}(\mu)\<E|\delta_{j_1},\cdots,\delta_{j_{\ell(\mu)}}\>_{\Gamma_+}^{\bullet,\bar{\tilde X}^+,Z}\\
   &   &\qquad\cdot\<\tau_{d_1}p^*\alpha_1,\cdots,\tau_{d_m}p^*\alpha_m|\delta^{j_1},\cdots,\delta^{j_{\ell(\mu)}}\>_{\Gamma_-}^{\bullet,\tilde X,E}.\nonumber
\end{eqnarray}
 Recall that we have assumed that
$$
     \dim \overline M_{g,m+1}(\tilde X,p^!A-e) =\sum\limits_{i=1}^mdeg\alpha_i+2\sum\limits_{i=1}^{m}d_i+2.
$$

Assume that a term with $(\Gamma^+,\Gamma^-)$ in RHS of (\ref{4-2}) has nonzero contribution. Then
\begin{eqnarray*}
    \dim \overline M_{\Gamma_+}(\bar{\tilde X}^+,Z)  & = &   2\int_{(p^!A-e)^+}c_1(\bar{\tilde X}^+)+2+2\ell(\mu)-2|\mu|, \\
    \dim \overline M_{\Gamma_-}(\tilde X,E) & = & \sum\limits_{i=1}^mdeg\alpha_i+2\sum\limits_{i=1}^{m}d_i+\sum\limits_{i=1}^{\ell(\mu)}deg\delta^{j_i}.
\end{eqnarray*}
So by the dimension constraint \eqref{dimension} for the degeneration formula, we have
\begin{eqnarray*}
         \frac{1}{2}\sum\limits_{i=1}^{\ell(\mu)}\textrm{deg}\delta^{j_i}+\int_{(p^!A-e)^+}c_1(\tilde X^+)-|\mu|=\ell(\mu).
\end{eqnarray*}
Let $\xi^+$ be the tautological line bundle of $\bar{\tilde X}^+=\bP_E(N_E\oplus\cO_E)$. Then Euler exact sequence gives
\begin{eqnarray*}
c_1(\bar{\tilde X}^+)=\pi^*c_1(E)+\pi^*c_1(N_E)-2c_1(\xi^+),
\end{eqnarray*}
where $\pi:\bar{\tilde X}^+\rightarrow E$ is the natural projection. Note that $N_E$ is the tautological line bundle of $E\cong\bP_C(N_C)$, and so
\begin{eqnarray*}
c_1(E)=\pi_E^*c_1(X)|_C-2c_1(N_E),
\end{eqnarray*}
where $\pi_E:E\rightarrow C$ is the natural projection. Therefore,
\begin{eqnarray*}
c_1(\bar{\tilde X}^+)=(\pi_E\circ\pi)^*c_1(X)|_C-\pi^*c_1(N_E)-2c_1(\xi^+).
\end{eqnarray*}
Note that we have the following natural decomposition
\begin{eqnarray*}
H_2(\bar{\tilde X}^+,\bZ)\cong\bZ F\oplus H_2(E,\bZ),
\end{eqnarray*}
and we can write
\begin{eqnarray*}
(p^!A-e)^+=aF+\pi_*(p^!A-e)^+,\textrm{ for some }a\in\bZ_{\geqslant 0}.
\end{eqnarray*}
We have the following constraints for $(p^!A-e)^+$:
\begin{eqnarray*}
\left\{\begin{array}{ccl}(p^!A-e)^+\cdot Z&=&|\mu|,\\(p^!A-e)^+\cdot E&=&(p^!A-e)\cdot E=1,\end{array}\right.
\end{eqnarray*}
and this gives
\begin{eqnarray*}
\pi_*(p^!A-e)^+\cdot E=-(|\mu|-1).
\end{eqnarray*}
Note that $-c_1(\xi^+)$ is the Poincar\'e dual of the divisor $Z$ in $\bar{\tilde X}^+$, and therefore
\begin{eqnarray*}
\int_{(p^!A-e)^+}c_1(\bar{\tilde X}^+)=\int_{(\pi_E\circ\pi)_*(p^!A-e)^+}c_1(X)|_C + 3|\mu|-1.
\end{eqnarray*}
Hence the dimension constraint becomes
\begin{eqnarray*}
\frac{1}{2}\sum\limits_{i=1}^{\ell(\mu)}deg\delta^{j_i}+\int_{(\pi_E\circ\pi)_*(p^!A-e)^+}c_1(X)|_C+2|\mu|=1+\ell(\mu).\nonumber
\end{eqnarray*}
Since $m\geqslant 1$, it follows that $\mu\neq\emptyset$ by the connectedness of the stable maps to $\tilde X$.
So the dimension constraint holds only if
\begin{eqnarray*}
\mu=(1),\quad deg\delta^{j_1}=0,\quad\int_{(\pi_E\circ\pi)_*(p^!A-e)^+}c_1(X)|_C=0,
\end{eqnarray*}
which implies Lemma \ref{curvelemma2}.
\end{proof}

Using  the above comparison results, the same argument as in the proof of Lemma \ref{ptlemma3} shows that the following lemma holds.

\begin{lemma}\label{curvelemma3}
Under the same assumptions as in Theorem \ref{curve}, denote $\langle[C],\tau_{d_1}\a_1,\cdots$ ,
$\tau_{d_m}\a_m\rangle^X_{g,A}$ and
$\langle\tau_{d_1}p^*\a_1,\cdots,\tau_{d_m}p^*\a_m\rangle^{\tilde X}_{g,p^!A-e}$ by
$H_g$ and $P_g$ respectively. Then
\begin{equation}\label{4-3}
      H_g=\sum_{g_1+g_2=g}C_{g_1}P_{g_2},
\end{equation}
where $C_g$'s can be determined by relative invariants ${\langle [C]|[pt]\rangle}^{\bar X^+,Z}_{g,F,(1)}$
 and ${\langle|[pt]\rangle}_{g,F,(1)}^{\bar{\tilde X}^+,Z}$.
\end{lemma}

  Similar to Theorem \ref{pt1}, we only need to determine the universal coefficients $C_g$'s in (\ref{4-3}). For this, we choose $X=\bP_C(N_C\oplus\cO_C), m=1, \alpha_1=[pt], A=F$ and rewrite
(\ref{4-3}) as follows
\begin{equation}\label{4-4}
   \langle[C],[pt]\rangle_{g,F}^{\bP_C(N_C\oplus\cO_C)}=\sum_{g_1+g_2=g}C_{g_1}\cdot \langle[pt]\rangle_{g_2,F}^{\bP_E(N_E\oplus\cO_E)}.
\end{equation}

About the two absolute invariants in (\ref{4-4}), we have

\begin{lemma}\label{curvelemma4}
\ban
\langle[C],[pt]\rangle_{g,F}^{\bP_C(N_C\oplus\cO_C)}&=&\frac{(-1)^g}{(2g+1)!\cdot 2^{2g}},\\
\langle[pt]\rangle_{g,F}^{\bP_E(N_E\oplus\cO_E)}&=&\delta_{g,0}.
\nan
\end{lemma}
\begin{proof}
In the first equality, let $C=\bP_C(\{0\}\oplus\cO_C)$ and $P_0\in\bP_C(N_C\oplus\{0\})$ be the Poincar\'e duals of $[C]$ and $[pt]$  respectively. There is a unique connected smooth embedded curve with homology class $F$ passing through $C$ and $P_0$, which is the line, in the fiber containing $P_0$, passing through $P_0$ and the intersection of $C$ and the fiber. Now LHS is equal to the degenerate contribution of the line. So Theorem 1.5 in \cite{Z} can be specialized to the first equality. The proof of the second equality is similar.
\end{proof}
\begin{remark}
Theorem 1.5 in \cite{Z} is the symplectic version of degenerate contribution computation for Fano case in \cite{P1}.
\end{remark}

{\bf Proof of Theorem \ref{curve}:} Using Lemma \ref{curvelemma4} and solving the equaiton (\ref{4-4}), we obtain the universal coefficients $C_g=\frac{(-1)^g}{(2g+1)!\cdot 2^{2g}}$, which gives Theorem \ref{curve}.

\begin{remark}
If $\int_Ac_1(X)>1$, then we can relax the condition $m>0$ to $m\geqslant 0$. One can check this by going through the proof of Lemma \ref{curvelemma1} and \ref{curvelemma2}.
\end{remark}

\section{Generalized BPS numbers}

Gromov-Witten invariants are only rational numbers in general, and hidden integrality for these invariants of projective $3$-folds has been studied since the very beginning of Gromov-Witten theory. For example, the mathematically non-rigorous computation of genus zero invariants of quintic $3$-folds \cite{COGP} inspired the famous multiple covering formula \cite{AM}. Based on M-theory consideration, Gopakumar and Vafa \cite{GV1,GV2} conjectured that countings of BPS states give hidden integrality for Gromov-Witten invariants of Calabi-Yau $3$-folds in all genera, which are multiplicities of certain representations of $SL(2)$ in the cohomology of moduli space of sheaves. Based on degenerate contribution computation, Pandharipande \cite{P1,P2} generalized the working definition of BPS numbers to arbitrary $3$-folds, and he also conjectured that these generalized BPS numbers are integers which are counts of curves satisfying incidence conditions given by insertions.

Let us review the definition of generalized BPS numbers and Pandharipande\rq{}s conjecture. Let $X$ be a connected closed symplectic manifold of real dimension $6$ and $A\in H_2(X,\bZ)$ a nonzero class. Note that by dimension consideration, $A$ carries nonzero Gromov-Witten invariants only if $\int_Ac_1(X)\geqslant 0$. Suppose that $\alpha_1,\cdots,\alpha_m\in H^{>2}(X,\bQ)$. When $\int_Ac_1(X)>0$, the generalized BPS number $n_{g,A}^X(\alpha_1,\cdots,\alpha_m)$ (at least one insertion) is given by
\ban
\sum\limits_{g=0}^\infty u^{2g}\<\alpha_1,\cdots,\alpha_m\>_{g,A}^X=\sum\limits_{g=0}^\infty u^{2g}n_{g,A}^X(\alpha_1,\cdots,\alpha_m)\cdot(\frac{\sin\frac{u}{2}}{\frac{u}{2}})^{2g-2+\int_Ac_1(X)},
\nan
and when $\int_Ac_1(X)=0$, then generalized BPS number $n_{g,A}^X$ (no insertion) is given by
\ban
\sum\limits_{g=0}^\infty u^{2g}\<\>_{g,A}^X=\sum\limits_{g=0}^\infty u^{2g}n_{g,A}^X\cdot\sum\limits_{\substack{d\in\bZ_{>0}\\\frac{A}{d}\in H_2(X,\bZ)}}\frac{1}{d}(\frac{\sin\frac{du}{2}}{\frac{u}{2}})^{2g-2}.
\nan
In general, the generalized BPS numbers are defined to satisfy the divisor equation and defined to vanish if degree $0$ and $1$ classes are inserted. So these invariants can be extended to include all cohomology classes.

If $X$ is a projective $3$-fold, and $\alpha_i$ is the Poincar\'e dual of a subvariety $X_i\subset X$ in general position, then Pandharipande conjectured that $n_{g,A}^X(\alpha_1,\cdots,\alpha_m)$ is the number of irreducible embedded curves in $X$ of geometric genus $g$, with homology class $A$ and intersecting all $X_i$'s. An important corollary of this conjecture is the integrality of generalized BPS numbers, which was proved by Zinger in the Fano case \cite{Z}, and by Ionel and Parker in the Calabi-Yau case \cite{IP3}.

{\bf Proof of Proposition \ref{BPS}:} We first prove Part (a) in Proposition \ref{BPS}.

From Corollary \ref{generating function}, we have
\ban
   \sum_{g=0}^\infty u^{2g-2}\langle[pt],\alpha_1,\cdots,\alpha_m\rangle^X_{g,A} = (\frac{\sin(u/2)}{u/2})^2 \sum_{g=0}^\infty u^{2g-2}\langle p^*\alpha_1,\cdots, p^*\alpha_m\rangle^{\tilde{X}}_{g, p^!A-e}.
\nan
Now by the definition of generalized BPS numbers, we have
\begin{eqnarray*}
  & &  \sum_{g\geq 0} u^{2g-2}n^X_{g,A}([pt],\alpha_1, \cdots, \alpha_m) (\frac{\sin (u/2)}{u/2})^{2g-2+\int_Ac_1(X)}\\
 & &\\
 &=& \sum_{g=0}^\infty u^{2g-2}\langle[pt], \alpha_1,\cdots,\alpha_m\rangle^X_{g,A} \\
 & &\\
 & = & (\frac{\sin(u/2)}{u/2})^2 \sum_{g=0}^\infty u^{2g-2}\langle p^*\alpha_1,\cdots, p^*\alpha_m\rangle^{\tilde{X}}_{g, p^!A-e}\\
 & &\\
 &=& (\frac{\sin(u/2)}{u/2})^2\sum_{g=0}^\infty u^{2g-2}n^{\tilde{X}}_{g,p^!A-e}(p^*\alpha_1,\cdots, p^*\alpha_m)(\frac{\sin(u/2)}{u/2})^{2g-2+\int_{p^!A-e}c_1(\tilde{X})}\\
 & &\\
 &=& \sum_{g=0}^\infty u^{2g-2}n^{\tilde{X}}_{g,p^!A-e}(p^*\alpha_1,\cdots, p^*\alpha_m) (\frac{\sin(u/2)}{u/2})^{2g-2+\int_Ac_1(X)}.
\end{eqnarray*}
This gives Part (a) of Corollary \ref{BPS}. The proof of Part (b) is analogous.

\begin{remark}
In the proof above, we only consider the case $m\geqslant 1$. The case $m=0$ can be treated similarly.
\end{remark}

{\bf Acknowledgements.}
The authors would like to thank Prof. Pandharipande for pointing out his early results about the degenerate contributions to us. We would also like to thank Prof. Yongbin Ruan and Pedro Acosta for their useful comments on earlier drafts.  Huazhong would like to thank Prof. Jian Zhou for sharing his ideas on Gromov-Witten theory generously, and Xiaowen Hu and Hanxiong Zhang for helpful discussions.  Weiqiang and Huazhong would like to thank Department of Mathematics of University of Michigan for its hospitality during their visiting.


\begin{thebibliography}{999}

\bibitem[AM]{AM}Aspinwall, P. S., Morrison, D. R., Topological field theory and rational curves, Comm. Math. Phys. 151 (1993), no. 2, 245-262.

\bibitem[B]{B} Behrend, K., Gromov-Witten invariants in algebraic geometry, Invent. Math. 127(1997), 601-617.

\bibitem[C]{C}Clemens, C. H., Degeneration of K\"ahler manifolds, Duke Math. J., 44 (1977), no. 2, 215-290.

\bibitem[COGP]{COGP}Candelas, P., de la Ossa, X. C., Green, P. S., Parkes, L., A pair of Calabi-Yau manifolds as as exactly soluble superconformal theory, Nuclear Phys. B 359 (1991), no. 1, 21-74.

\bibitem[FO]{FO}Fukaya, K., Ono, K., Arnold conjecture and Gromov-Witten invariant, Topology,38(5)(1999), 933-1048.

\bibitem[FP]{FP} Faber, C., Pandharipande, R., Hodge integrals and Gromov-Witten theory, Invent. Math. 139 (2000), 173-199.

\bibitem[G]{G} Gathmann, A., Gromov-Witten invariants of blow-ups, J. Algebraic Geom. 10 (2001), no. 3, 399-432.

\bibitem[Gi]{Gi} Givental, A., Equivariant Gromov-Witten invariants, Internat. Math. Res. Notices (1996), 613-663.

\bibitem[GV1]{GV1}Gopakumar, R., Vafa, C., M-theory and topological strings I, hep-th/9809187.

\bibitem[GV2]{GV2}Gopakumar, R., Vafa, C., M-theory and topological strings II, hep-th/9812127.

\bibitem[GP]{GP}Graber, T., Pandharipande, R., Localization of virtual classes, Invent.math. 135, 487-518(1999).

\bibitem[GV]{GV}Graber, T.,  Vakil, R., Relative virtual localization and vanishing of tautological classes on moduli spaces of curves, Duke Math. J. 130 (1) (2005) 1-37.

\bibitem[H1]{H1}Hu, J., Gromov-Witten invariants of blow-ups along points and curves, Math.Z. 233, 709-739(2000).

\bibitem[H2]{H2}Hu, J., Gromov-Witten invariants of blow-ups along surfaces, Compositio Math. 125 (2001), no. 3, 345-352.

\bibitem[HLR]{HLR} Hu, J., Li, T.-J., Ruan, Y., Birational cobordism invariance of uniruled symplectic manifolds, Invent. Math., 172(2008), 231-275.

\bibitem[IP]{IP}Ionel, E., Parker, T., The Symplectic Sum Formula for Gromov-Witten Invariants, Ann. of Math., 159(3), 2004, 935-1025.

\bibitem[IP3]{IP3}Ionel, E., Parker, T., The Gopakumar-Vafa formula for symplectic manifolds, arXiv:1306.1516v2.

\bibitem[Ko1]{Ko1}Kontsevich, M., Intersection theory on the moduli space of curves and the matrix Airy function, Comm. Math. Phys., 147(1992), no. 1, 1-23.

\bibitem[Ko2]{Ko2}Kontsevich, M.,  Enumeration of rational curves via torus actions, in The Moduli Space of Curves, Progress in Math. 129 Boston-Basek-Berlin, 1995, 335-368.

\bibitem[Li]{Li}Li, J., A degeneration formula of GW-invariants, J. Diff. Geom., 60(2002),199-293.

\bibitem[LLY]{LLY}Lian, B., Liu, K., Yau, S.-T., Mirror principle I, Asian J. Math. 1(1997),729-763.

\bibitem[LR]{LR}Li, A.-M., Ruan, Y., Symplectic surgery and Gromov-Witten invariants of Calabi 3-folds, Invent.Math. 145 (2001), no. 1, 151-218.

\bibitem[LT1]{LT1}Li, J., Tian, G., Virtual moduli cycles and Gromov-Witten invariants of algebraic varieties, J. Amer. Math. Soc., 11(1)(1998), 119-174.

\bibitem[LT2]{LT2}Li, J., Tian, G., Virtual moduli cycles and Gromov-Witten invariants of general symplectic manifolds, Topics in symplectic 4-manifolds(Irvine, CA, 1996), 47-83, First Int. Press Lect. Ser. I, Internat. Press, Cambridge, MA, 1998.

\bibitem[MP]{MP}Maulik, D., Pandharipande, R., A topological view of Gromov-Witten theory, Topology 45 (2006) 887-918.

\bibitem[MS]{MS}McDuff, D., Salamon, D., J-holomorphic curves and symplectic topology, AMS Colloquium Publications, vol. 52.

\bibitem[P1]{P1}Pandharipande, R., Hodge integrals and degenerate contributions, Commum. Math. Phys. 208(1999), 489-506.


\bibitem[P2]{P2}Pandharipande, R., Three questions in Gromov-Witten theory, International Congress of Mathematics, Vol. II (Beijing, 2002), 503-512, Higher Ed. Press, Beijing, 2002.

\bibitem[Q]{Q}Qi, X., A blowup formula of high-genus Gromov-Witten invariants of symplectic 4-manifolds, to appear in Advances in Mathematics (China).

\bibitem[R1]{R1}Ruan, Y., Sympletic topology on algebraic 3-folds, J. Diff. Geom., 39(1994), 215-227.

\bibitem[R2]{R2}Ruan, Y., Virtual neighborhoods and pseudo-holomorphic curves, Proceedings of 6th Gokova Geometry-Topology Conference, Turkish J. Math., 23(1)(1999),161-231.

\bibitem[R3]{R3}Ruan, Y., Surgery, quantum cohomology and birational geometry, Northern California Symplectic Geometry Seminar AMS Translations, Series 2, 1999 (196), 183-198.

\bibitem[RT1]{RT1}Ruan, Y., Tian, G., A mathematical theory of quantum cohomology, J. Diff. Geom., 42(2)(1995), 259-367.

\bibitem[RT2]{RT2}Ruan, Y., Tian, G., Higher genus symplectic invariants and sigma model coupled with gravity, Invent. Math. 130(1997), 455-516.

\bibitem[S]{S}Siebert, B., Gromov-Witten invariants for general symplectic manifolds, preprint.


\bibitem[W]{W}Witten, E., Two-dimensional gravity and intersection theory on moduli space, Surveys in differential geometry (Cambridge, MA, 1990), 243-310, Lehigh Univ., Bethlehem, PA, 1991.

\bibitem[Z]{Z}Zinger, A., A comparison theorem for Gromov-Witten invariants in the symplectic category, Adv. Math., 228 (2011), no. 1, 535-574.

\end{thebibliography}
\end{document}